\documentclass[12pt, reqno]{amsart}

\numberwithin{equation}{section}
\usepackage{amsfonts, amssymb, latexsym, epsfig, amsthm, enumitem}
\usepackage{hyperref}
\usepackage[capitalise, noabbrev]{cleveref}
\usepackage[usenames,dvipsnames]{xcolor}
\usepackage{graphicx,pgf,tikz}
\usetikzlibrary{decorations.markings}
\usepackage{ytableau}
\usepackage[all]{xy}
\usepackage{wrapfig}
\usepackage{mathdots}
\usepackage{dirtytalk}
\usepackage{soul}

\newlist{thmlist}{enumerate}{1}
\setlist[thmlist]{label=(\roman{thmlisti}),noitemsep}
\usepackage[top=1in, bottom=1in, left=1in, right=1in]{geometry}

\usepackage{enumitem}

\newtheorem{theorem}{Theorem}[section]
\newtheorem*{theorem*}{Theorem}
\newtheorem{proposition}[theorem]{Proposition}

\newtheorem*{claim*}{Claim}

\newtheorem{Main Conjecture}[theorem]{Main Conjecture}

\theoremstyle{definition}
\newtheorem{defn}[theorem]{Definition}
\newtheorem{examplex}[theorem]{Example}
\newenvironment{example}
  {\pushQED{\qed}\examplex}
  {\popQED\endexamplex}

\newtheorem{remark}[theorem]{Remark}
\newtheorem{notation}[theorem]{Notation}
\theoremstyle{plain}

\newcommand{\init}{{\tt in}}
\newcommand{\hgt}{{\rm{ht }}}

\newcommand{\ZZ}{\mathbb{Z}}

\newcommand{\NN}{\mathbb{N}}

\newcommand{\G}{\mathcal{G}}
\newcommand{\ba}{\mathbf{a}}
\newcommand{\bb}{\mathbf{b}}
\newcommand{\bc}{\mathbf{c}}
\newcommand{\bx}{\mathbf{x}}

\newcommand{\cF}{\mathcal F}
\newcommand{\cT}{\mathcal T}

\renewcommand{\P}{\mathcal{P}}

\title[A case study in Gorenstein Liaison]{Multicomplex configurations\\ A case study in  Gorenstein Liaison}
\author[Klein]{Patricia Klein}
\address{Department of Mathematics, Texas A\&M University, College Station, TX, USA}
\email{pjklein@tamu.edu}
\author[Rajchgot]{Jenna Rajchgot}
\address{Department of Mathematics and Statistics, McMaster University, Hamilton, ON, Canada}
\email{rajchgoj@mcmaster.ca}
\author[Seceleanu]{Alexandra Seceleanu}
\address{Mathematics Department, University of Nebraska--Lincoln, Lincoln, NE, USA.}
\email{aseceleanu@unl.edu}

\thanks{Klein was partially supported by NSF DMS-2246962. Rajchgot was partially supported by NSERC Discovery Grants 2017-05732 and 2023-04800. Seceleanu was partially supported by NSF DMS-2401482.}

\date{\today}
\begin{document}
\maketitle

\begin{abstract}
    We introduce and investigate multicomplex configurations, a class of projective varieties constructed via specialization of the polarizations of Artinian monomial ideals. Building upon geometric polarization and geometric vertex decomposition, we establish conditions under which such configurations retain desirable algebraic properties. In particular, we show that, given suitable choices of linear forms for substitution, the resulting ideals admit Gröbner bases with prescribed initial ideals and are in the Gorenstein liaison class of a complete intersection. 
\end{abstract}

\section{Introduction}
The paper \cite{MatroidConfigurations} by Geramita, Harbourne, Migliore and Nagel introduces and develops the concepts of {\em $\lambda$-configurations} and {\em matroid configurations} by considering on one hand unions of all complete intersections of a fixed codimension among a set of given hypersurfaces and on the other ideals obtained from the Stanley-Reisner ideal of a matroid by substitution of the variables with a regular sequence of forms. This allows one to establish close connections between $\lambda$-configurations or matroid configurations and Stanley-Reisner ideals of matroids, leveraging  liaison theory to describe their Hilbert function, minimal generators, and symbolic powers. %A key result is that all symbolic powers of the ideal of any matroid configuration are Cohen-Macaulay, significantly extending prior results known only for specific monomial ideals and point configurations.

This work leaves open the problem of determining the relationship between a Stanley-Reisner ideal or more generally a monomial ideal and the ideal resulting from it by substitution of the variables with arbitrary forms, which do not necessarily form a regular sequence  nor satisfy uniformity conditions.  Specifically, it is not necessarily the case in this paper that any subset of $c+1$ of the forms cuts out a codimension $c+1$ variety, as is the case in the theory of $\lambda$-configurations. 

In this paper we consider geometric configurations obtained by specializing monomial ideals defining Artinian quotient rings. In \cite{M2011}, Murai refers to  such a monomial ideal  as the ideal of a multicomplex. We adopt this terminology here and extend it to introduce the notion of a multicomplex configuration (see \Cref{def:multicomplex}). This refers to projective varieties defined as intersections of reducible hypersurfaces, whose components are chosen from a given collection according to combinatorial rules determined by a multicomplex. A multicomplex configuration can also be viewed as a specialization of the polarization of a multicomplex. Simplical complexes arising from multicomplexes by polarization are called grafted complexes, studied by Faridi in \cite{Fa2005}.

In the generality of \Cref{def:multicomplex}, multicomplex configurations appear to be new, although a subset of these termed distractions have been well-studied as indicated below.

The idea of polarizing a monomial ideal and then specializing the variables of the resulting polarization to linear forms goes back to Hartshorne's proof of the connectedness of the Hilbert scheme in \cite{Har66}. The resulting ideals are called {\em distractions} and are studied further by Bigatti, Conca and Robbiano in \cite{BCR}. Therein conditions are given on the linear forms that ensure that the operation of distraction preserves Betti numbers and results in an arithmetically Cohen-Macaulay variety.

A common feature of the Stanley-Reisner ideals of matroids complexes and of the ideals of multicomplexes  is that they are in the Gorenstein liaison class of a complete intersection (abbreviated {\em glicci}). In the case of matroids this is shown by Nagel and R\"omer in \cite[Theorem 3.3]{NR08} applying the combinatorial technique of vertex decomposition. For the polarizations of multicomplexes the vertex decomposable property is established in \cite[Remark 1.8]{M2011} and the glicci property in \cite[Theorem 3.16]{FKRS}. A careful reading of the proof of \cite[Proposition 2.3]{MatroidConfigurations}, which employs basic double linkage, reveals that matroid configurations are also glicci. 

The current paper focuses on conditions under which distraction of multicomplexes preserves the Gr\"obner basis and the glicci property. This is in contrast with arbitrary unions of linear subspaces of projective space, which can sometimes fail to be Cohen-Macaulay, hence also fail to be glicci. Our main theorem can be abstracted as follows

\begin{theorem*}[\Cref{thm: main}]
Consider a multicomplex $M$, the polynomial ring $S=\kappa[x_1, \ldots, x_m]$, a monomial order $<$ on $S$ refining $x_1>x_2>\cdots>x_m$, and
collections $\cF_1, \ldots, \cF_n$ of linear forms in $S$ such that the initial monomial of each element in $\cF_i$ is $x_i$ and the ideals listed below are pairwise distinct
$$
\left(\ell_{1,b_1}, \ldots, \ell_{n,b_n}\right) \text{ for all }\bb\in\NN^n \text{ such that } \bx^\bb \in M.
$$
Then 
\begin{itemize}
\item the minimal generators of the ideal $I(M)$ of the multicomplex become, upon substituting variables with linear forms from $\cF$, a Gr\"obner basis for the ideal $I(M,\cF)$ of the multicomplex and
\item the ideal $I(M,\cF)$, and hence the corresponding multicomplex configuration, are glicci.
\end{itemize}
\end{theorem*}

To prove this theorem we leverage ideas from works on geometric polarization and geometric vertex decomposition.
 In~\cite{KMY09} Knutson, Miller and Yong introduced  geometric vertex decomposition to study ideals in a polynomial ring and corresponding algebraic varieties. Instead of using Gr\"obner degeneration to pass directly to a monomial ideal, which can lead to a complicated degeneration which forgets too much of the indexing combinatorics too quickly, a step-by-step degeneration process is considered. 
 
%\PJK{I changed this part a little bit, mostly because the limits KMY consider all do happen to be reduced, so the issue for them was losing track of the indexing combinatorics if they went straight to the initial ideal rather than a failure of reducedness.}
 
 Given a reduced variety $X$, the process of geometric vertex decomposition involves rescaling one coordinate axis at a time, breaking $X$ into projection and cone parts analogous to the link and deletion in a vertex decomposition. If the degeneration remains reduced, key invariants like multidegrees and Hilbert series can be computed separately for these parts and later combined. By iterating this procedure across all coordinate axes, one can eventually reach the full Gr\"obner degeneration while retaining inductive insights. When the final result consists of coordinate subspaces (equivalently, is defined by a squarefree monomial ideal), this process aligns with the classical vertex decomposition in simplicial complexes.

%\PJK{I really like your description of geometric vertex decomposition here!}

Generalizing the relationship between vertex decomposition of monomial ideals and Gorenstein liaison presented in  \cite{NR08},  Klein and Rajchgot established a  connection between Gorenstein  liaison and geometric vertex decomposition in \cite{KR21}. Building upon the classical operation on monomial ideals termed polarization, geometric polarization (\Cref{def:geoPol}) was introduced by the authors in \cite{FKRS} as a means to 
apply the technique of geometric vertex decomposition in conjunction with Gorenstein biliaison in new contexts. This technique, detailed in \cref{s: geom pol} is what allows us to control the outcome of  substituting collections of linear forms with carefully chosen leading terms into the ideal of a multicomplex. Thus we view our work as a case study in Gorenstein liaison and an invitation to discovering new contexts where this technique can be brought to bear and new properties of multicomplex configurations.

\section{Multicomplex Configurations}

Inspired by \cite{MatroidConfigurations} and \cite{M2011}, in this work we introduce the notion of a multicomplex configuration. Throughout this paper, let $\kappa$ denote an arbitrary field.  We work over a pair of polynomial rings $R=\kappa[x_1,\ldots, x_n]$ and $S=\kappa[x_1,\ldots, x_m]$ with $m\geq n$, which are related by the inclusion $R\subseteq S$. 

\begin{defn}[\cite{M2011}]
Given a vector $\bc\in \NN^n$, a {\em $\bc$-monomial} is a monomial $x_1^{a_1}x_1^{a_2}\cdots x_n^{a_n}$ such that $a_i\leq c_i$ for $1\leq i\leq n$. 
A {\em $\bc$-multicomplex $M$} is a non-empty, finite set of {\em $\bc$-monomials}  such that if $u \in M$ and a monomial $v$ divides $u$ then $v \in M$. The {\em ideal $I(M)$ of a $\bc$-multicomplex $M$} is the ideal generated by the monomials which do not belong to $M$. 
\end{defn}

By definition a $\bc$-multicomplex $M$ is a finite down-set in the partially ordered set of monomials in $R=\kappa[x_1,\ldots, x_n]$ with respect to divisibility. It forms a vector space basis for the quotient ring $R/I(M)$, therefore demonstrating that this ring is indeed Artinian. The multicomplex $M$ consists of the standard monomials of this quotient ring.  Conversely, every monomial ideal defining an Artinian quotient corresponds to a  $\bc$-multicomplex $M$ for infinitely many vectors $\bc$.

\begin{defn}\label{def:multicomplex}
Given a vector $\bc\in \NN^n$, a $\bc$-multicomplex $M$, and $n$ ordered lists of homogeneous polynomials $\cF_i=\{f_{i,j}\in S=\kappa[x_1,\ldots, x_m] :  0\leq j\leq c_i\}$ with $1\leq i\leq n$, let $\cF = \bigcup_{1 \leq i \leq n} \cF_i$.  We define the  {\em multicomplex configuration} determined by the given data to be the vanishing locus in projective space of the ideal 
\begin{equation}\label{eq: multicomplex}
I(M,\cF)=
\left( \prod_{i=1}^n\prod_{j=0}^{a_{i}-1}f_{i,j}:  \prod_{i=1}^n x_i ^{a_i}\in G(I(M)) \right),
\end{equation}
where $G(I(M))$ denotes the (unique) minimal monomial set of generators of $I(M)$.
\end{defn}

\begin{remark}
    If $g=\prod_{i=1}^nx_i ^{a_i}$ is a minimal monomial generator of $I(M)$ it follows that for all  indices $i$ we have $g/x_i\in M$ and thus $a_i-1\leq c_i$ for all $i$. This justifies that the formula in \eqref{eq: multicomplex} is well-defined, that is, $f_{i,j}\in\cF_i$ for all $0\leq j\leq a_i-1$.
\end{remark}

Another way to construct $I(M,\cF)$ in two steps is to first polarize the monomial ideal $I(M)$, obtaining a squarefree monomial ideal $I(M)^{\rm pol}$ of $R^{\rm pol}=\kappa[x_{i,j}: 1\leq i\leq n, 1\leq j\leq c_i]$ and then define $I(M,\cF)$ as the image of $I(M)^{\rm pol}$ via the ring homomorphism $R^{\rm pol}\to S$ determined by $x_{i,j}\mapsto f_{i,j}$. In this manner, multicomplex configurations fit in the paradigm of specializing Stanley-Reisner ideals described in the introduction.

\begin{example}\label{ex-not-lambda-or-matroid}
Consider the $(2,1)$-multicomplex $M=\{1, x_1, x_1^2, x_2\}$ with corresponding ideal $I(M)= (x_1^3, x_1x_2, x_2^2)\subseteq k[x_1,x_2]$. We construct a multicomplex configuration utilizing the linear forms $\cF_1=\{x_1, x_1+x_2+3x_3, x_1+x_3\}$ and $\cF_2=\{x_2+x_3,x_2\}$ in $S=k[x_1,x_2, x_3]$. This results in 
\[
I(M,\cF)=(x_1(x_1+x_2+3x_3)(x_1+x_3), x_1(x_2+x_3),(x_2+x_3)x_2).
\]
The primary decomposition of the above ideal is
\[
I(M,\cF)=(x_1, x_2+x_3)\cap (x_1+x_2+3x_3, x_2+x_3)\cap (x_1+x_3, x_2+x_3)\cap (x_1,x_2),
\]
which reveals that our multicomplex configuration is a set of four points in $\mathbb{P}^2$. Since a $\lambda$-configuration in $\mathbb{P}^2$ would consist of a binomial number $\binom{r}{2}$ of points for some $r>0$, this is not a $\lambda$-configuration.

Moreover, the polarized ideal $I(M)^{\rm pol}$ is not the Stanley-Reisner ideal of a matroid since it can be verified computationally that its second symbolic power is not Cohen-Macaulay, while symbolic powers of Stanley-Reisner ideals of matroids are always Cohen-Macaulay by \cite{Varbaro}, \cite{MinhTrung}.
\end{example}

In \cite[Remark 1.8 and Lemma 3.2]{M2011} Murai shows that, for any multicomplex $M$, the ideal $I(M)^{\rm pol}$ is vertex decomposable, from which it follows via \cite[Theorem 3.3]{NR08} that the same ideal is glicci.  Under suitable conditions on $\cF$, we wish to extend these properties to the multicomplex configuration ideals $I(M,\cF)$, replacing vertex decomposability by geometric vertex decomposability, a notion we recall below.

\section{Geometric polarization and Gorenstein liaison}
\label{s: geom pol}

The following sections make consistent use of term orders, initial ideals, and Gr\"obner bases.  For a review of standard facts and terminology, we refer  to \cite[Chapter 15]{Eis95}.

\subsection{Geometric polarization}

Geometric polarization was introduced in 
\cite{FKRS} with the purpose of extending the polarization operation from monomial ideals to arbitrary ideals of a polynomial ring.

\begin{defn}\label{def:geoPol}
Fix a variable $y = x_j$ of $S = \kappa[x_1, \ldots, x_m]$.  For $g \in S$, write $g=\sum_{i=0}^t y^ir_i$, $r_i \in S$, where $r_t \neq 0$ and $y$ does not divide any term of any $r_i$.  Using a new variable $y'$, define $$\mathcal{P}_y(g) = r_0+yr_1+\sum_{i=2}^t y(y')^{i-1}r_i \in S[y'].$$ For an ideal $I$, term order $<$, and $<$-Gr\"obner basis $\mathcal{G}$, define $\mathcal{P}_y(\mathcal{G})=\{\mathcal{P}_y(g) \colon g \in \mathcal{G}\}$.  We call $\mathcal{P}_y(\mathcal{G})$ the \bf{one-step geometric polarization of $\mathcal{G}$ with respect to $y$}. 
\end{defn}

If $I$ is a monomial ideal and $\mathcal{G}$ is its set of minimal monomial generators, a sequence of one-step polarizations ultimately yields the minimal monomial generators for the polarization of $I$ in the usual sense.  Like in the monomial case, we have that $(\mathcal{G}, y-y') = (\mathcal{P}_y(\mathcal{G}), y-y')$. If $\mathcal{G}$ consists of monomials, then $y-y'$ is a regular element on $R/(\mathcal{G})$; however, this condition no longer necessarily holds if $\mathcal{G}$ consists of arbitrary polynomials. In \Cref{thm:inducedGB} we review how this condition relates to $\mathcal{P}_y(\mathcal{G})$ being a Gr\"obner basis under term orders that are compatible with geometric vertex decomposition, a notion we now review.

\subsection{Geometric vertex decomposition}

Knutson, Miller, and Yong \cite{KMY09} introduced geometric vertex decomposition in their study of vexillary matrix Schubert varieties, otherwise known as one-sided mixed ladder determinantal ideals.

\begin{defn}\label{def:ycomp}
Suppose that $S = \kappa[x_1, \ldots, x_m]$ is equipped with a term order $<$.  Fix $j \in [n]$, and set $y = x_j$. For a polynomial $f = \sum_{i=0}^t \alpha_i y^i \in S$, where $\alpha_t \neq 0$ and no term of any $\alpha_i$ is divisible by $y$, write $\init_{y} (f) = \alpha_t y^t$.  If $\init_<(f) = \init_<(\init_{y}(f))$ for all $f \in S$, we say that $<$ is a \textbf{$y$-compatible term order}.  

For an ideal $I$ of $S$, write $\init_{y}(I) = (\init_{y}(f) \colon f \in I)$.  If there exists a $<$-Gr\"obner basis $\mathcal{G}$ of $I$ such that $\init_<(g) = \init_<(\init_{y}(g))$ for all $g \in \mathcal{G}$, we call $<$ \textbf{$y$-compatible with respect to $I$}.   If $<$ is $y$-compatible with respect to $I$, then it follows from \cref{def:ycomp} that $\init_<(I) = \init_<(\init_{y}(I))$.
\end{defn}

For $0 \neq f \in S$, let $\deg_y(f)$ denote the greatest power of $y$ that divides at least one term of $f$. In particular, $\deg_y(f) = 0$ if $f\neq 0$ and no term of $f$ is divisible by $y$.  Define $\deg_y(0) = -\infty$.

\begin{notation}\label{not:indOrder} Let $S = \kappa[x_1, \ldots, x_m]$ and $S' = S[y']$.  Fix a term order $<$ on $S$.  If $\mu$ is a monomial in $S'$, let depol$(\mu)$ denote the monomial of $S$ obtained by replacing in $\mu$ each $y'$ with a $y$. Define a term order $\prec$ on $S'$ in the following way: $\mu \prec \nu$ if depol$(\mu)<$depol$(\nu)$ or if depol$(\mu)=$depol$(\nu)$ and $\deg_{y}(\mu)<\deg_y(\nu)$.
\end{notation}

The following theorem elucidates the circumstances under which the geometric polarization of a Gr\"obner basis is a Gr\"obner basis.

\begin{theorem}[{\cite[Theorem 5.7]{FKRS}}]
\label{thm:inducedGB}
    Fix a homogeneous ideal $I$ of $S = \kappa[x_1, \ldots, x_m]$, a term order $<$ that is $y$-compatible with respect to $I$, and a $<$-Gr\"obner basis $\mathcal{G}$ of $I$.  Then $\mathcal{P}_y(\mathcal{G})$ is a $\prec$-Gr\"obner basis of $(\mathcal{P}_y(\mathcal{G}))$ if and only if $y-y'$ is a nonzerodivisor on $S[y']/(\mathcal{P}_y(\mathcal{G}))$.
    %Furthermore, if $\mathcal{G}$ is a reduced Gr\"obner basis and $\mathcal{P}_y(\mathcal{G})$ a Gr\"obner basis, then it is a reduced as well.
\end{theorem}

Although \cite[Theorem 5.7]{FKRS} is stated for a $y$-compatible term order $<$, the proof only requires that $<$ be $y$-compatible with respect to $I$.

%\PJK{I moved this down so that $y$-compatible would be defined before this theorem is stated.}

We now arrive to the notion of geometric vertex decomposition, following the presentation in \cite[Section 2.1]{KMY09} and \cite[Section 2]{KR21}. 

%\PJK{I added the self citation because KMY doesn't talk about degenerate or non-degenerate and may or may not want to be saddled with it (and may or may not approve of the notation that Jenna and I chose for a lot of this).}

\begin{proposition}[{\cite[Theorem 2.1]{KMY09}}]\label{prop:GB linear in y gives gvd}
Fix an ideal $I$ and a term order $<$ that is $y$-compatible with respect to $I$. Suppose that $I$ has  a Gr\"obner basis of  the form 
\[
\mathcal{G} = \{yq_1+r_1, \dots, yq_k+r_k, h_1, \ldots, h_\ell\},
\]
where no term of any $q_i$, $r_i$, or $h_i$ is divisible by $y$ and define 
\begin{eqnarray*}
I' &=& (yq_1, \dots, yq_k, h_1, \ldots, h_\ell), \\
    C_{y,I} &=& (q_1, \dots, q_k, h_1, \ldots, h_\ell), \\
    N_{y,I} &=& (h_1, \ldots, h_\ell).
\end{eqnarray*}
%\begin{eqnarray*}
 %   C_{y,I} &=& \bigcup_{i \geq 1} (\init_{y} (I): y^i) \\
 %   N_{y,I} &=& (I \cap \kappa[x_1, \ldots, \hat{y}, \ldots , x_n])R.
%\end{eqnarray*}
 %Note that $N_{y,I}+(y) = \init_{y}(I)+(y)$ and that $C_{y,I}$ and $N_{y,I}$ each have a generating set that does not involve $y$.  
 Then 
 \begin{enumerate}
     \item  $I'=\init_{y}(I)$ and the given generating sets of $I'$, $C_{y,I}$, and $N_{y,I}$ are $<$-Gr\"obner bases.
 \item
$\init_{y}(I) = yC_{y,I}+N_{y,I} = C_{y,I} \cap (N_{y,I}+(y))$.
\end{enumerate}
\end{proposition}

\begin{defn}
    In the setup described above, we call the equation $\init_{y}(I) = C_{y,I} \cap (N_{y,I}+(y))$ a \textbf{geometric vertex decomposition of $I$ with respect to $y$.}  We call a geometric vertex decomposition \textbf{nondegenerate} if $C_{y,I} \neq (1)$ and $\sqrt{C_{y,I}} \neq \sqrt{N_{y,I}}$. In this context the {\bf geometric link} of $I$ is  $C_{y,I}$ and the {\bf geometric deletion} of $I$ is $N_{y,I}$.
\end{defn}

This is in analogy with the notion of vertex decomposition for simplicial complexes. Algebraically, vertex decomposition translates into decomposing a monomial ideal in terms of the ideals of the link and deletion of a vertex similarly to the formulas in item (2) of \Cref{prop:GB linear in y gives gvd}.

%The ability to iteratively apply geometric vertex decompositions leads to the idea of geometrically vertex decomposable ideals.

%\begin{defn}\label{def:gvd}
%An ideal $I\subseteq S = \kappa[x_1, \ldots, x_m]$ is \textbf{geometrically vertex decomposable} if $I$ is unmixed and if
%\begin{enumerate}
%\item $I = ( 1)$ or $I$ is generated by a (possibly empty) set of indeterminates in $S$, or
%\item for some variable $y = x_j$ of $S$, 
%$\text{in}_y I =  C_{y,I}  \cap (N_{y,I}+ (y))$ is a geometric vertex decomposition and the contractions of $N_{y,I}$ and $C_{y,I}$ to $\kappa[x_1,\dots,\widehat{y},\dots, x_m]$ are geometrically vertex decomposable.
%\end{enumerate}
%\end{defn}

\subsection{Gorenstein biliaison}
Let $S$ be a ring and $I$ an ideal of $S$. We say that $S/I$ is \textbf{generically Gorenstein}, or $\mathbf{G_0}$, if $(S/I)_P$ is Gorenstein for every minimal prime $P$ of $I$. We say that $I$ is \textbf{unmixed} if $\hgt(I) = \hgt(P)$ for every associated prime $P$ of $I$.   Recall that a Cohen--Macaulay ideal is necessarily unmixed.

 We recall the definition of Gorenstein linkage of homogeneous ideals of a polynomial ring.

\begin{defn}\label{def:GorensteinLikage}
    Let $S$ be a polynomial ring over a field, and let $I$ and $J$ be homogeneous, saturated ideals of $S$. We say that $I$ and $J$ are \textbf{directly Gorenstein linked (or G-linked)} by a homogeneous, saturated ideal $C$ if $C \subseteq I \cap J$, $S/C$ is Gorenstein, $C:I = J$, and $C:J = I$.  The equivalence relation generated by $G$-linkage
 is called Gorenstein liaison. An ideal in the Gorenstein liaison equivalence class of a complete intersection is called  {\bf glicci}. \end{defn}

%As the name suggests, if $C$ is a basic double G-link of $B$, then $B$ is G-linked to $C$ in two steps \cite[Proposition 5.10]{KMM+01}, i.e., there is an ideal $D$ so that $B$ is directly G-linked to $D$ and $D$ is directly G-linked to $C$. 

%\begin{defn}\label{def:basic double G-link} 
%Let $R$ be a polynomial ring over a field, and let $f$ be a homogeneous degree $d$ element of $R$.  Let $A\subset B$ be homogeneous, unmixed proper ideals. Then the ideal $fB+A$ is called a \textbf{basic double G-link of $B$ on $A$ of shift $d$} if  
%\begin{itemize}
%    \item $R/A$ is Cohen--Macaulay and $G_0$;
%    \item $\hgt(A) = \hgt(B)-1$; and
 %   \item $A:f=A$.
%\end{itemize}
 %  If $A$ and $B$ are monomial ideals and $f$ is a monomial, we call $f B+A$ a \textbf{monomial basic double G-link of $B$ on $A$}. 

%Abusing notation, we will sometimes refer to an equality $C = fB+A$ as a basic double G-link to mean that $C$ is a basic double G-link of $B$ on $A$.
%\end{defn}

%For a discussion of basic double G-linkage in terms of linear equivalence of divisors, see \cite[Section 4]{KMM+01}.

Geometric vertex decomposition is intimately related to elementary G-biliaison, which we review now. 
 \begin{defn}\label{def:Gbiliaison}
Let $I$ and $D$ be homogeneous, saturated, unmixed ideals of $S = \kappa[x_1, \ldots, x_m]$.  Suppose there exists a homogeneous ideal $N$ satisfying all of the following properties:
\begin{itemize}
    \item $N \subseteq I \cap D$ ;
    \item $S/N$ is Cohen--Macaulay and $G_0$;
    \item $\hgt(C) = \hgt(N)+1=\hgt(I)$;
    \item there is an isomorphism of graded $S/N$-modules $I/N \cong [D/N](-\ell)$ for some $\ell\in \ZZ$.
\end{itemize}
Then we say that $I$ is obtained from $D$ by an \textbf{elementary G-biliaison of shift $\ell$}.  
\end{defn}

As the name suggests, if $I$ is obtained from $D$ by an elementary G-biliaison, as described in \Cref{def:Gbiliaison}, then $I$ is G-linked to $D$ in two steps, i.e., there is an ideal $J$ so that $I$ is directly G-linked to $J$ and $J$ is directly G-linked to $D$ \cite[Theorem 3.5]{Har07}. See also \cite[Proposition 5.12]{KMM+01} and \cite[Proposition 5.10]{KMM+01} for prior related results.  

%\AS{What is a reference for this? Tricia, can you suggest some way to handle this so that we can say that ideals related by G-biliaison are in the same Gorenstein linkage class -- the following is the reference for basic double links}
%\PJK{KMMR+ essentially did elementary G-biliaison but had a $G_1$ hypothesis that Hartshorn was able to weaken to a $G_0$ hypothesis.  There weirdly isn't a clean citation I'm aware of for the claim that elementary G-biliaison (in Hartshorne's sense) and basic double G-link generate the same equivalence class.  That fact is relevant to people's interest in elementary G-biliaison but not relevant to proofs in our paper, so maybe we don't say anything more about that than point to the existence of prior related work.}

A main contribution of \cite{FKRS} is to relate geometric polarization to elementary G-biliaison in the manner described in the following result. Note that the hypothesis on the Gr\"obner basis in the following theorem is weaker than that of \Cref{prop:GB linear in y gives gvd} in that it does not assume the elements of $\G$ are linear in the variable $y$. This assumption is replaced by requiring that the polarization $\mathcal{P}_y(\G)$, which is by construction linear in $y$, be a Gr\"obner basis.

\begin{theorem}[{\cite[Theorem 5.16]{FKRS}}]\label{thm:polarLink}
    Let $S = \kappa[x_1, \ldots, x_m]$, and let $I$ be an unmixed and homogeneous proper ideal of $S$.  Fix $i \in [n]$, and write $y=x_i$.  Fix a term order $<$ that is $y$-compatible with respect to $I$, and suppose $I$ admits a $<$-Gr\"obner basis 
    of the form \[
\mathcal{G} = \{yd_1+r_1, \ldots, yd_k+r_k, h_1, \ldots, h_\ell\},
    \]  where $y$ does not divide any term of any $r_j$ or any $h_j$.  Suppose that $y-y'$ is a nonzerodivisor on $S[y']/(\mathcal{P}_y(\mathcal{G}))$.  Set $D = (d_1, \ldots, d_k, h_1, \ldots, h_\ell)$ and $N = (h_1, \ldots, h_\ell)$, and let $D'$ be the ideal of $S[y']$ obtained from $D$ by the substitution $y \mapsto y'$.  Then the following are equivalent:
    \begin{enumerate}
        \item There is an elementary G-biliaison $(D/N)[-1] \cong I/N$, where the isomorphism is given by multiplication by any of the $(yd_i+r_i)/d_i$ for $i\in [k]$, all of which are equivalent modulo $N$;
        \item There is an elementary G-biliaison $(D'/NS[y'])[-1] \cong (\mathcal{P}_y(\mathcal{G}))/NS[y']$, where the isomorphism is given by multiplication by some $v'/d'$ with $v' \in (\mathcal{P}_y(\mathcal{G}))$ and $d' \in D'$ nonzerodivisors on $S[y']/NS[y']$ with $\init_y(v')/d' = y$;
        \item $(\mathcal{P}_y(\mathcal{G}))$ is unmixed and admits a nondegenerate geometric vertex decomposition at $y$, the geometric link, $D$, is unmixed, and the geometric deletion, $N$, is a Cohen--Macaulay and $G_0$ ideal.
    \end{enumerate}
\end{theorem}

Although \cite[Theorem 5.16]{FKRS} is stated for a $y$-compatible term order $<$, the proof only requires that $<$ be $y$-compatible with respect to $I$.

%\AS{Should we just keep (2) and (3) since this is the equivalence we actually use in our proof?}
%\PJK{I guess in the unlikely event that someone comes across this here and wants to use it, it's nice for them also to know what the isomorphism looks like?  But I don't feel strongly - whatever you think is best.}

\section{Main Theorem}

%{\color{blue}Mention ``distraction" of Huneke-Ulrich, etc.}\AS{Did Huneke-Ulrich use distractions? Where? I included a discussion of distractions as appear in Hartshorne in the introduction.}

%\PJK{Maybe I'm not right to say distraction here.  The theorem I had in mind that we should relate to is \cite[Theorem 4.2]{HU07}, where they also have as input an artinian ideal and then construct some polynomials along the way.  We differ in that they actually study the monomial ideal itself in the ring it naturally lives in and also in that they don't get to insist on homogeneous links.  So, same flavor but different outcomes.}

%\PJK{I no longer believe these \cite{HU07} is close enough for there to be a smooth way for us to discuss it.}

While a priori when defining multicomplex configurations in \Cref{def:multicomplex} we impose no restrictions of the members of the collection $\cF$ of forms used for substitution, to study the properties of $I(M,\cF)$ using Gorenstein liaison theory methods developed in \cite{FKRS}, particularly \cref{thm:polarLink}, we need to control the initial monomials of the elements of $\cF$ with respect to a given term order. The following condition aids in this regard.

\begin{defn}
Fix a monomial order $<$ on $S=\kappa[x_1,\ldots, x_m]$. We say a set $T$ of homogeneous elements of $S$  is {\em $x_i$-initial} if the leading monomial of each $f\in T$ is $\init_<(f)=x_i$. 
\end{defn}

For a homogeneous element $f\in S$, $\init_<(f)=x_i$ implies $f$ is a linear form.

%\PJK{I don't think we need ``If $<$ is a refinement of the degree partial order" because $f$ is already assumed homogeneous.}

%In this section we consider $c\leq n-1$ families of hyperplanes $\H_1, \ldots, \H_c$ in projective space. Each family $\H_i$ consists of $s_i$ hyperplanes defined by linear forms $\ell_{i,j}\in\kappa[x_1,\ldots, x_n]$ with $0\leq j <s_i$ such that with respect to a suitable monomial order $<$ they satisfy $\init_<(\ell_{i,j})=x_j$. We also consider a set $E$ of tuples $\ba_i=(a_{i1},\ldots, a_{ic})\in \NN^c$ that are pairwise incomparable under the partial order on $\NN^c$ given by coordinate-wise comparison.

%Two ideals are determined by the given data: a monomial ideal $J(E)=(\bx^\ba \mid \ba \in E)$ of $S=\kappa[x_1,\ldots, x_n]$, whose minimal generators are $\bx^\ba:=\prod_{j=1}^nx_j^{a_{i,j}}$ and an ideal generated by products of linear forms 
%\begin{equation}\label{eq:idealProductsLF}
%I(\H_1, \ldots, \H_c, E):=\left( \prod_{i=1}^c\prod_{j=0}^{a_{i,j}-1}\ell_{i,j}:  \ba \in E\right).
%\end{equation}

%Under mild conditions on $E$ and on the intersection of members of the each family, we show that this class of ideals is glicci.

We now arrive at our main result.  As usual, take $R =\kappa[x_1,\ldots, x_n] \subseteq \kappa[x_1,\ldots, x_m] = S$.

%\PJK{I wasn't quite sure where the right place to declare the rings was, but we do need them both named before (6).  This is one option?}

%\PJK{As written before, it was possible for $<$ not to be $x_1$-compatible.  There are a couple of ways to solve this.  One is to define ``$y$-compatible with respect to $I$" where you only evaluate $\init_<(f) = \init_<(\init_y(f))$ on a Gr\"obner basis of $I$ and then note that everything we cite for $y$-compatible term orders goes through in this slightly more general setting.  Another option is to just give this theorem for lex orders.  Currently, I've chosen the lazy version, but I'm happy to do the more general thing if you think that's better.}

%\AS{I think I like the ``$y$-compatible with respect to $I$" solution better.}

\begin{theorem}\label{thm: main}
If $\bc\in \NN^n$, $M$ is a $\bc$-multicomplex, $<$ is monomial order on $S$ refining $x_1>x_2>\cdots>x_m$, and for each $ i\in [n]$ the set $\cF_i=\{\ell_{ij}: 0\leq j\leq c_i\}$ is an $x_i$-initial set of linear forms of $S$ such that the collection of ideals

\[
\left(\ell_{1,b_1}, \ldots, \ell_{n,b_n}\right) \text{ for all } \bb\in\NN^n \text{ such that } \bx^\bb \in M 
\]

are distinct, then the multicomplex configuration ideal $I(M,\cF)$ has the following properties:
\begin{enumerate}
\item  the irredundant primary decomposition of $I(M,\cF)$ is
\[
I(M,\cF)=\bigcap\limits_{\bb\in\NN^n,\bx^\bb \in M}\left(\ell_{1,b_1}, \ldots, \ell_{n,b_n}\right),
\]
thus  the multicomplex configuration $M(\cF)$ is a union of $|M|$ linear spaces of codimension $n$, 
\item the set of generators of $I(M,\cF)$ described in \eqref{eq: multicomplex} is a Gr\"obner basis,
\item
the initial ideal of $I(M,\cF)$ is the ideal $I(M)$ of the muticomplex extended to $S$, i.e.,
\[
\init_< I(M,\cF)=I(M)S,
\]
\item $I(M,\cF)$ is Cohen-Macaulay of $\hgt I(M,\cF)=n$ and degree $\deg(S/I(M,\cF))=|M|$,
\item $I(M,\cF)$ is  glicci,
\item the Betti numbers of $I(M)$ and $I(M,\cF)$ coincide, that is, $\beta_{ij}^R(I(M))=\beta_{ij}^S(I(M,\cF))$.
\end{enumerate}
\end{theorem}

\begin{proof}
Recall that $I(M)$ is the ideal of $R$ generated by monomials not in $M$. Let $E$ be the set of exponents of the minimal monomial generators of $I(M)$.

For $\ba\in E$, denote
$f(\ba)=\prod_{i=1}^n\prod_{j=0}^{a_{i}-1}\ell_{i,j}$.

Then one can write the definition of the multicomplex configuration ideal in \eqref{eq: multicomplex} as
\begin{equation}\label{eq: linear multicomplex}
I(M,\cF)=
\left( f(\ba): \ba\in E \right).
\end{equation}
 We will now make a series of observations and small computations, (a)-(e), which we will collect to prove items (1)-(6) of the present theorem.

%\PJK{I inserted the above to avoid error of associating (a) with (1), etc. and then getting confused. And then some other blazes on the trail later down.}
 
    \begin{enumerate}
\item[(a)] \[
I(M)S\subseteq \init_< I(M,\cF).
\]
Indeed, for any minimal generator $\bx^{\ba}$ of $I(M)S$, using the hypothesis that the sets $\cF_i$ are $x_i$-initial, one deduces $\bx^{\ba} = \prod_{i=1}^n\prod_{j=0}^{a_{i}-1}x_i =\prod_{i=1}^n\prod_{j=0}^{a_{i}-1}\init(\ell_{i,j})=\init(f(\ba)) \in \init_< I(M,\cF)$.

\item[(b)] 
\begin{equation}\label{eq: J}
I(M,\cF)\subseteq J:=\bigcap\limits_{\bb\in\NN^n,\bx^\bb \in M}\left(\ell_{1,b_1}, \ldots, \ell_{n,b_n}\right).
\end{equation}

Indeed, the condition $\bx^\bb \in M$ is satisfied if and only if for each $\ba\in E$ there exists an index $j$ such that $0\leq b_j<a_{j}$,  yielding that $\ell_{i,b_j}\mid f(\ba)$ and thus that each generator $f(\ba)$ of $I(M,\cF)$ described in \eqref{eq: linear multicomplex} belongs to the ideal $J$.

\item[(c)] \[\hgt(I(M,\cF))=n\]
Recall that $R/I(M)$ is Artinian, which is equivalent to $\hgt(I(M))=n$. From item (a) one has $n=\hgt(I(M))\leq \hgt(\init_< I(M,\cF))=\hgt( I(M,\cF))$. From item (b) it follows that $\hgt( I(M,\cF))\leq \hgt(J)=n$.  The two inequalities yield the desired equality.

\item[(d)] \[
\deg(S/J)=\deg(R/I(M)).
\] According to the hypothesis, the ideals in the intersection \eqref{eq: J} are distinct. Each of these ideals is a complete intersection of height $n$ as it is generated by a Gr\"obner basis consisting of linear forms and in particular the listed generators form a linearly independent set. By the associativity formula, it follows that the degree of $J$ is equal to the number of ideals in the intersection, which is the cardinality of the multicomplex $M$. Since the cardinality of the multicomplex $M$ is also equal to $\deg(R/I(M))$, we have shown that $\deg(S/J)=\deg(R/I(M))$.

\item[(e)]
\begin{eqnarray}\label{eq: inequalities}
    \deg(S/I(M)S) &\geq& \deg(S/\init_< I(M,\cF))=\deg(S/ I(M,\cF)) \nonumber \\
    &\geq & \deg(S/J)=\deg(R/I(M))=\deg(S/I(M)S)
\end{eqnarray}
The inequalities in the above chain follow from items (a) and (b), respectively, observing that the ideals whose degrees are being compared have the same height. The first equality follows from Gr\"obner theory, the second is item (d), and the third is standard.

%\PJK{(d) used to be part of (b).  I hope this shuffling is okay.}

%follows from the equality $\deg(S/J)=|M|$ obtained in (b) and the fact that the degree of $R/I(M)$ is equal to the length of this Artinian ring, which is equal in turn to the cardinality of $M$.

We are now prepared to draw the desired conclusions from items (a)-(e).

\emph{Proof of (1)-(4):} An important consequence of the string of inequalities \eqref{eq: inequalities} is that all must be equalities. Since $I(M)S$ is an unmixed ideal, $I(M)S\subseteq \init_< I(M,\cF)$ and these two ideals have the same height and same degree, it follows as claimed in (3) that $I(M)=\init_< I(M,\cF)$; see \cite[Lemma 8]{Engheta}. Consequently the set $\G=\{f(\ba): \ba\in E\}$ forms a Gr\"obner basis for $I(M,\cF)$, settling (2).  Since its initial ideal is Cohen-Macaulay, we deduce that $I(M,\cF)$ is Cohen-Macaulay \cite[Proposition 25.4]{Peeva}, proving item (4). From the equality of their multiplicities in \eqref{eq: inequalities} it now follows by the same argument as above that $I(M,\cF)=J$, settling (1).

We now exhibit a geometric vertex decomposition for $I(M,\cF)$ to establish (5).

    The proof proceeds by double induction on $n$ and for  each fixed $n$ on the cardinality of $M$. If $n=1$, then the set $E=\{e\}$ consists of only one element  and the ideal $I(M,\cF)=\left( \prod_{j=1}^{e} \ell_{1,j}\right)$ is principal. 
      Assume that $n>1$.  The smallest multicomplex is $M=\{1\}$ with $I(M)=(x_1,\ldots,x_n)$ and $I(M,\cF)$ an ideal generated by $c$ linearly independent linear forms, which is a complete intersection.  

     % {\color{blue}
      If $x_1\in I(M)$ then $M$ is a multicomplex supported on $x_2, \ldots, x_n$.  Hence $I(M, \mathcal{F}) = I(M, \mathcal{F} \setminus \{\ell_{0,1}\})+(\ell_{0,1})$, where $\ell_{0,1}$ is a linear form in a variable not involved in $I(M, \mathcal{F} \setminus \{\ell_{0,1}\})$. %\AS{What is $\ell_{0,1}$ here?} \PJK{Thanks! Should have been $\ell_{0,1}$.} 
      By induction on $n$, $I(M, \mathcal{F} \setminus \{\ell_{0,1}\})$ is glicci via links within $\kappa[x_2, \ldots, x_m]$.  If $G = I(M, \mathcal{F} \setminus \{\ell_{0,1}\}):J$ is a direct G-link of ideals of $\kappa[x_2, \ldots, x_m]$, then $G+(\ell_{0,1}) = I(M, \mathcal{F}):(J+(\ell_{0,1}))$ is a direct G-link in $S$.  Hence, $I(M, \mathcal{F})$ is also glicci.

    %  I don't know if I've said too much or too little.  It does not currently feel clean, though.}

 %     If $x_1\in I(M)$ then $M$ is a multicomplex supported on $x_2, \ldots, x_n$, meaning that $M\subseteq \overline{R}=\kappa[x_2, \ldots, x_n]$ and $S/I(M,\cF)\cong \overline{S}/\overline{I}$, where $\overline{S}=S/(\ell_{1,0})\cong\kappa[x_2,\ldots, x_{m}]$ and $\overline{I}=I(M,\cF\setminus\{\ell_{1,0}\})$ is the ideal of the matroid complex of $M$ in $\overline{S}$.
%By the inductive hypothesis on $n$, $\overline{I}$ is glicci, thus so is $I(M,\cF)$.

%{\color{blue}PJK: If $\ell_{1,0} \in I$, then perform a (degenerate) gvd at $x_1$ and note that the link is $(1)$ and the deletion is gvd by induction.}
%{\color{blue}PJK: We can consider pointing out that this would also follow from Juan and Uwe's CM+$G_0$+contains a linear form => glicci.}

Assume from now on that $n>1$ and $x_1\not \in I(M)$. Note, using (2), that $<$ is $x_1$-compatible with respect to $I$.
%with respect to the given  monomial order  $y$ is the largest among $x_1,\ldots, x_c$ so that   
Let $x_1'$ be a new variable.  For $j\geq 1$ set $\ell'_{1,j}\in S[x_1']$ to be obtained from $\ell_{1,j}$ by replacing $x_1$ by $x_1'$. Recall from (2) that 
\[
\G=\{f(\ba) :\ba\in E\}=\left \{\prod_{i=1}^n\prod_{j=0}^{a_{i}-1}\ell_{i,j}:  \ba\in E \right \}
\]
is a Gr\"obner basis for $I(M,\cF)$. 
Its one-step geometric polarization with respect to $x_1$ (\cref{def:geoPol}) is the set
\[ \P_{x_1}(\G)=\left \{\ell_{1,0}^{\min\{a_{1},1\}} \prod_{j=1}^{a_{1}-1}\ell'_{1,j}\prod_{i=2}^n\prod_{j=0}^{a_{i}-1}\ell_{i,j} :\ba\in E \right \}.
\]
Set $\cT_1=\{\ell_{1,0}\}, \cT'_1=\{\ell'_{1,1}, \ldots, \ell'_{1,c_1}\}$, $\cT_i=\cF_i$ for $i\geq 2$, and $\cT = \cT_1' \cup \bigcup_{1 \leq i \leq n} \cT_i$. By (2), $\P_{x_1}(\G)$ is a Gr\"obner basis for the ideal $I(M, \cT)\subseteq S[x_1']$ it generates with respect to the monomial order $\prec$ in \Cref{not:indOrder}.  Thus by \cref{thm:inducedGB} we conclude that $x_1-x_1'$ is regular on $S[x_1']/(\P_{x_1}(\G))$.

Set 
\begin{eqnarray*} N&=&\left(\prod_{i=2}^n\prod_{j=0}^{a_{i}-1}\ell_{i,j} :\ba\in E, a_{1}=0\right),  \\
D&=&\left(\prod_{j=1}^{a_{1}-1}\ell'_{1,j}\prod_{i=2}^n\prod_{j=0}^{a_{i}-1}\ell_{i,j} :\ba\in E, a_{1}>0\right)+N,  
\text{ and }\\
\cF' &=& \cF_2 \cup \cdots \cup \cF_n,\\
\cF^\ast_1 &=&\{\ell'_{1,1}, \ldots, \ell'_{1,c_1}\},\\
\cF'' &=& \cF^\ast_1 \cup \cF_2  \cdots \cup \cF_n.
\end{eqnarray*}
Let $M'=\{u\in M: x_1\nmid u\}$.  Because $M'$ is supported on $\{x_2, \ldots, x_n\}$ and the linear forms in $\cF'$ do not involve $x_1$, we may construct $I(M',\cF')$ as an ideal of $\kappa[x_2, \ldots, x_n]$.  Then $N=I(M',\cF')S[x_1']$.    %By (4), we know that $I(M',\cF')$ is Cohen-Macaulay of height $n$.  Hence, $N$ is Cohen-Macaulay of height $n-1$. Because $I(M',\cF')$ is radical by (1), so too is $N$; hence $N$ is $G_0$.

Let $M''$ be the multicomplex corresponding to the ideal $I(M'')=I(M):x_1$.  Observe that $D$ is obtained from $I(M'',\cF'')$ by replacing every $x_1$ with a $x_1'$.  Because $D$ does not involve $x_1$ and $I(M'',\cF'')$ does not involve $x_1'$, this is simply a matter of renaming a variable.  Since $I(M'')=I(M):x_1$ is a proper subset of $I(M)$, the multicomplex $M''$ has cardinality strictly smaller than $M$, thus we may apply the inductive hypothesis to $I(M'',\cF'')$ and apply the conclusions also to $D$. %is Cohen--Macaulay, hence unmixed, of height $n$.
%cardinality and respectively support strictly smaller than $M$. 
%(By the support of $M'$ we mean the set $x_2, \ldots, x_n$ given that $M'\subseteq \kappa[x_2, \ldots, x_n]$.)  According to the inductive hypothesis $N$ and $D$ are both Cohen-Macaulay of heights $\hgt(D)=n=\hgt(N)+1$. Moreover $N$ is $G_0$ because the  the primary components of $N$ are ideals  generated by $n-1$ linear forms by (1). 

 By \Cref{prop:GB linear in y gives gvd} we see that, since $\P_{x_1}(\G)$ is a Gr\"obner basis whose elements are at most linear in $x_1$, $(\P_{x_1}(\G))$  admits a geometric vertex decomposition with respect to $x_1$. Moreover, the geometric link is $D$, which is unmixed by (4), and the geometric deletion is $N$, which is Cohen-Macaulay and $G_0$ by  (1) and (4). \cref{thm:polarLink} now allows to conclude that  $I(M,\cF)$ and $D$ are linked via an elementary $G$-biliaison.

This settles claim (5), as $D$ is glicci in view of the inductive hypothesis.

Item (6) follows either from \cite[Corollary 2.20]{BCR} or from \cite[Proposition 5.10]{FKRS}. Although this item is not new, we include it here for completeness and to illustrate another analogy with the theory of matroid configurations; see \cite[Theorem 3.3]{MatroidConfigurations}. \qedhere
    \end{enumerate}
    \end{proof}

In order to assist in parsing the notation of \Cref{thm: main}, we provide an example, which is a continuation of \Cref{ex-not-lambda-or-matroid}.

    \begin{example}
Recall the $(2,1)$-multicomplex $M=\{1, x_1, x_1^2, x_2\}$ with corresponding ideal $I(M)= (x_1^3, x_1x_2, x_2^2)\subseteq k[x_1,x_2]$ and multicomplex configuration utilizing the linear forms $\cF_1=\{x_1, x_1+x_2+3x_3, x_1+x_3\}$ and $\cF_2=\{x_2+x_3,x_2\}$ in $S=\kappa[x_1,x_2, x_3]$, so that
\[
I(M,\cF)=(x_1(x_1+x_2+3x_3)(x_1+x_3), x_1(x_2+x_3),(x_2+x_3)x_2).
\]

In this case \begin{align*}
\mathcal{P}_{x_1}(\mathcal{G}) &= (x_1(x_1'+x_2+3x_3)(x_1'+x_3), x_1(x_2+x_3),(x_2+x_3)x_2), \\
N &= ((x_2+x_3)x_2),\\
D &= ((x_1'+x_2+3x_3)(x_1'+x_3), x_2+x_3,(x_2+x_3)x_2) = ((x_1'+x_2+3x_3)(x_1'+x_3), x_2+x_3),\\
\mathcal{F}' &= \mathcal{F}_2,\\
\mathcal{F}_1^\ast &= \{x_1'+x_2+3x_3, x_1'+x_3\},\\
\mathcal{F}'' &= \mathcal{F}_1^\ast \cup \mathcal{F}_2,\\
M' &= \{1,x_2\} \mbox{ and } I(M'), \mbox{ which we generate in $\kappa[x_2]$, is equal to } (x_2^2), \\
M'' &= \{1,x_1\} \mbox{ and } I(M'') = (x_1^2, x_2).
\end{align*} 
The reader may verify that $I(M',\mathcal{F}')S[x_1'] = N$ and that $D$ is obtained from $I(M'',\mathcal{F}'')$ by replacing each $x_1$ with at $x_1'$.
    \end{example}

  %  \AS{Should we say something about an alternate proof of part (5) by using basic double linkage to link the ideal $I(M,\cF)$ to a c.i. more straight forwardly by writing
%\[
%I(M,\cF)=\ell_{1,0}I(M',\cF_1\setminus\{\ell_{1,1}\}, \ldots, \cF_n)+I(M'',\cF_2, \ldots, \cF_n)
%\]
%with $M',M''$ as defined in the proof above or ignore this altogether? This would parallel the proof of \cite[Proposition 2.3]{MatroidConfigurations}, but with a different set of hypotheses.}

%%%%%%%%%%%%%%%%%%%%%%%%%%%%%%%%%%%%%%%%%

\bibliographystyle{amsalpha}
\bibliography{refs}

\newcommand{\etalchar}[1]{$^{#1}$}
\providecommand{\bysame}{\leavevmode\hbox to3em{\hrulefill}\thinspace}
\providecommand{\MR}{\relax\ifhmode\unskip\space\fi MR }
\providecommand{\MRhref}[2]{%
  \href{http://www.ams.org/mathscinet-getitem?mr=#1}{#2}
}
\providecommand{\href}[2]{#2}
\begin{thebibliography}{KMMR{\etalchar{+}}01}

\bibitem[BCR05]{BCR}
A.~M. Bigatti, A.~Conca, and L.~Robbiano, \emph{Generic initial ideals and
  distractions}, Comm. Algebra \textbf{33} (2005), no.~6, 1709--1732.
  \MR{2150838}

\bibitem[Eis95]{Eis95}
David Eisenbud, \emph{Commutative algebra}, Graduate Texts in Mathematics, vol.
  150, Springer-Verlag, New York, 1995, With a view toward algebraic geometry.
  \MR{1322960}

\bibitem[Eng07]{Engheta}
Bahman Engheta, \emph{On the projective dimension and the unmixed part of three
  cubics}, J. Algebra \textbf{316} (2007), no.~2, 715--734. \MR{2358611}

\bibitem[Far05]{Fa2005}
Sara Faridi, \emph{Cohen-{M}acaulay properties of square-free monomial ideals},
  J. Combin. Theory Ser. A \textbf{109} (2005), no.~2, 299--329.

\bibitem[FKRS25]{FKRS}
S.~Faridi, P.~Klein, J.~Rajchgot, and A.~Seceleanu, \emph{Polarization and
  gorenstein liaison}, Com. Amer. Math. Soc. (2025).

\bibitem[GHMN17]{MatroidConfigurations}
A.~V. Geramita, B.~Harbourne, J.~Migliore, and U.~Nagel, \emph{Matroid
  configurations and symbolic powers of their ideals}, Trans. Amer. Math. Soc.
  \textbf{369} (2017), no.~10, 7049--7066. \MR{3683102}

\bibitem[Har66]{Har66}
Robin Hartshorne, \emph{Connectedness of the {H}ilbert scheme}, Inst. Hautes
  \'{E}tudes Sci. Publ. Math. (1966), no.~29, 5--48. \MR{213368}

\bibitem[Har07]{Har07}
\bysame, \emph{Generalized divisors and biliaison}, Illinois J. Math.
  \textbf{51} (2007), no.~1, 83--98. \MR{2346188}

\bibitem[KMMR{\etalchar{+}}01]{KMM+01}
Jan~O. Kleppe, Juan~C. Migliore, Rosa Mir\'{o}-Roig, Uwe Nagel, and Chris
  Peterson, \emph{Gorenstein liaison, complete intersection liaison invariants
  and unobstructedness}, Mem. Amer. Math. Soc. \textbf{154} (2001), no.~732,
  viii+116. \MR{1848976}

\bibitem[KMY09]{KMY09}
Allen Knutson, Ezra Miller, and Alexander Yong, \emph{Gr\"{o}bner geometry of
  vertex decompositions and of flagged tableaux}, J. Reine Angew. Math.
  \textbf{630} (2009), 1--31. \MR{2526784}

\bibitem[KR21]{KR21}
Patricia Klein and Jenna Rajchgot, \emph{Geometric vertex decomposition and
  liaison}, Forum Math. Sigma \textbf{9} (2021), Paper No. e70, 23.
  \MR{4329854}

\bibitem[MT11]{MinhTrung}
Nguyen~Cong Minh and Ngo~Viet Trung, \emph{Corrigendum to
  ``{C}ohen-{M}acaulayness of monomial ideals and symbolic powers of
  {S}tanley-{R}eisner ideals'' [{A}dv. {M}ath. 226 (2) (2011) 1285--1306]
  [mr2737785]}, Adv. Math. \textbf{228} (2011), no.~5, 2982--2983. \MR{2838067}

\bibitem[Mur11]{M2011}
Satoshi Murai, \emph{Spheres arising from multicomplexes}, J. Combin. Theory
  Ser. A \textbf{118} (2011), no.~8, 2167--2184.

\bibitem[NR08]{NR08}
Uwe Nagel and Tim R\"{o}mer, \emph{Glicci simplicial complexes}, J. Pure Appl.
  Algebra \textbf{212} (2008), no.~10, 2250--2258. \MR{2426505}

\bibitem[Pee11]{Peeva}
Irena Peeva, \emph{Graded syzygies}, Algebra and Applications, vol.~14,
  Springer-Verlag London, Ltd., London, 2011. \MR{2560561}

\bibitem[Var11]{Varbaro}
Matteo Varbaro, \emph{Symbolic powers and matroids}, Proc. Amer. Math. Soc.
  \textbf{139} (2011), no.~7, 2357--2366. \MR{2784800}

\end{thebibliography}

\end{document}